\documentclass[11pt]{amsart}
\usepackage{geometry}                
\usepackage{fancyhdr}

\usepackage{amsmath}
\usepackage{mathrsfs}
\usepackage{graphicx}
\usepackage{amssymb}
\usepackage{epstopdf}
\usepackage[all]{xypic}
\usepackage{palatino}
\usepackage{color}

\usepackage{hyperref}
\usepackage[arrow,curve,matrix]{xy}

\renewcommand{\Im}{{\rm Im}}

\renewcommand{\dim}{{\rm dim}}

\newcommand{\coker}{{\rm coker}}

\newcommand{\Hom}{{\rm Hom}}

\theoremstyle{plain}

\newtheorem{thm}{Theorem}[section]

\newtheorem{cor}[thm]{Corollary}
\newtheorem{prop}[thm]{Proposition}
\newtheorem{lem}[thm]{Lemma}

\theoremstyle{definition}
\newtheorem{defn}[thm]{Definition}
\newtheorem{ex}[thm]{Example}
\newtheorem{rmk}[thm]{Remark}

\def\GG{{\textbf G}}

\def\cA{\mathcal{A}}

\def\cM{\mathcal{M}}
\def\cR{\mathcal{R}}

\def\cU{\mathcal{U}}

\newcommand{\bb}[1]{\mathbb{#1}}
\newcommand{\mc}[1]{\mathcal{#1}}
\newcommand{\mf}[1]{\mathfrak{#1}}
\newcommand{\ol}[1]{\overline{#1}}

\newcommand{\ul}[1]{\underline{#1}}
\newcommand{\defi}[1]{{\em #1}}
\newcommand{\op}[1]{\operatorname{#1}}

\renewcommand{\a}{\alpha}
\renewcommand{\b}{\beta}
\newcommand{\bw}{\bigwedge}
\renewcommand{\coker}{\operatorname{Coker}}

\newcommand{\im}{\operatorname{Im}}

\newcommand{\nc}{\operatorname{nc}}
\renewcommand{\ker}{\operatorname{Ker}}

\newcommand{\onto}{\twoheadrightarrow}
\newcommand{\oo}{\otimes}

\newcommand{\vnc}{\operatorname{vnc}}

\def\lra{\longrightarrow}

\newcommand{\Gr}{\operatorname{Gr}}

\newcommand{\IA}{\operatorname{IA}}
\newcommand{\OA}{\operatorname{OA}}

\newcommand{\Sym}{\operatorname{Sym}}

\theoremstyle{plain}

\newtheorem{lemma}[thm]{Lemma}
\newtheorem{proposition}[thm]{Proposition}

\theoremstyle{definition}

\newtheorem*{ack}{Acknowledgment}

\newtheorem*{thm-main*}{Main Theorem}

\begin{document}
\title[Topological invariants of groups and Koszul modules]{Topological invariants of groups and Koszul modules}


\author[M. Aprodu]{Marian Aprodu}
\address{Marian Aprodu: Simion Stoilow Institute of Mathematics
\hfill \newline\texttt{}
 \indent P.O. Box 1-764,
RO-014700 Bucharest, Romania, and \hfill  \newline\texttt{}
\indent  Faculty of Mathematics and Computer Science, University of Bucharest,  Romania}
\email{{\tt marian.aprodu@imar.ro}}

\author[G. Farkas]{Gavril Farkas}
\address{Gavril Farkas: Institut f\"ur Mathematik,   Humboldt-Universit\"at zu Berlin \hfill \newline\texttt{}
\indent Unter den Linden 6,
10099 Berlin, Germany}
\email{{\tt farkas@math.hu-berlin.de}}

\author[\c{S}. Papadima]{\c{S}tefan Papadima\textsuperscript{\textdagger} }
\address{\c{S}tefan Papadima: Simion Stoilow Institute of Mathematics \hfill \newline\texttt{}
\indent P.O. Box 1-764,
RO-014700 Bucharest, Romania}
\thanks{\c{S}tefan Papadima passed away on  January 10, 2018. }
\author[C. Raicu]{Claudiu Raicu}
\address{Claudiu Raicu: Department of Mathematics,
University of Notre Dame \hfill \newline\texttt{}
\indent 255 Hurley Notre Dame, IN 46556, USA, and \hfill\newline\texttt{}
\indent Simion Stoilow Institute of Mathematics, \hfill\newline\texttt{}
\indent  P.O. Box 1-764, RO-014700 Bucharest, Romania}
\email{{\tt craicu@nd.edu}}

\author[J. Weyman]{Jerzy Weyman}
\address{Jerzy Weyman: Institute of Mathematics,
Uniwersytet Jagiello\'nski \hfill \newline\texttt{}
\indent 30-348,
Krak\'ow, Poland}
\email{{\tt  jerzy.weyman@gmail.com}}

\begin{abstract}
We provide a uniform vanishing result for the graded components of the finite length Koszul module associated to a subspace $K\subseteq \bigwedge^2 V$, as well as a sharp upper bound for its Hilbert function. This purely algebraic statement has interesting applications to the study of a number of invariants associated to finitely generated groups, such as the Alexander invariants, the Chen ranks, or the degree of growth and virtual nilpotency class. For instance, we explicitly bound the aforementioned invariants in terms of the first Betti number for the maximal metabelian quotients of (1) the Torelli group associated to the moduli space of curves; (2) nilpotent fundamental groups of compact K\"ahler manifolds; (3) the Torelli group of a free group.\end{abstract}

\maketitle


%
%


\section{Introduction}

A well-established way of studying the properties of a finitely generated group $G$ is by considering its associated graded Lie algebra
$$\mbox{gr}(G):=\bigoplus_{q\geq 1} \Gamma_q(G)/\Gamma_{q+1}(G),$$
defined in terms of its \emph{lower central series}
\[G=\Gamma_1(G)\geq \ldots \geq \Gamma_q(G)\geq \Gamma_{q+1}(G) \geq \ldots\mbox{ where }\Gamma_{q+1}(G)=[\Gamma_q(G),G].\]
The quotients $\mbox{gr}_q(G)=\Gamma_q(G)/\Gamma_{q+1}(G)$ are finitely generated abelian groups and their ranks $\phi_q(G):=\mbox{rk } \mbox{gr}_q(G)$ encode fundamental numerical information about $G$. For example, $\Gamma_2(G)=G'$ is the commutator subgroup of $G$ and $\mbox{gr}_1(G)=G_{\mathrm{ab}}$ is the abelianization of $G$, whose rank $b_1(G)$ is the first Betti number of $G$. 

\vskip 3pt

For general $q\geq 2$, it is extremely difficult to compute $\phi_q(G)$. Following Chen \cite{Chen}, Massey \cite{Massey}, Papadima-Suciu \cite{PS-imrn} and many others, a  more manageable task is to pass instead to
the \emph{metabelian quotient} $G/G''$, where $G''=[G',G']$ is the second commutator subgroup of $G$, and to consider its graded Lie algebra $\mbox{gr}(G/G'')$
and the corresponding \emph{Chen ranks}
$$\theta_q(G):=\phi_q(G/G'')=\mbox{rk } \mbox{gr}_q(G/G'').$$
The first Chen rank always satisfies $\theta_1(G)=\phi_1(G)=b_1(G)$, hence our focus will be on the case $q\geq 2$, when we only have the inequality $\phi_q(G)\geq \theta_q(G)$. For basic properties of these invariants initially considered in knot theory we refer to \cite{PS-imrn}, whereas for applications of Chen ranks to hyperplane arrangements one can consult \cite{CS}, \cite{CSch} and the references therein. 

It turns out that that in many cases $\mbox{gr}(G/G'')$ is an object belonging essentially to commutative algebra, which can be studied to great effect with homological and algebro-geometric methods. The precise context where these methods are most fruitful is when $G$ is a \emph{$1$-formal group} in the sense of Sullivan \cite{Su77}. We recall that for instance fundamental groups of compact K\"ahler manifolds, hyperplane arrangement groups, or the Torelli group $T_g$ of the mapping class group are all known to be $1$-formal. We refer to \cite{PSform} for a survey of various facets of formality. 
We begin with a sample application of the results of our work (see Theorem~\ref{thm:bound-chen}).

\begin{thm}\label{thm:Kahler}
 Let $G$ be a finitely generated $1$-formal group and suppose its first Betti number is $n=b_1(G)\geq 3$. If $G/G''$ is nilpotent, then $\theta_q(G)=0$ for $q\geq n-1$ and
 \[\theta_q(G) \leq {n+q-3\choose n-1} \cdot \frac{(n-2)(n-1-q)}{q}, \quad\mbox{ for }q=2,\ldots,n-2.\]
\end{thm}

What is striking about Theorem~\ref{thm:Kahler} is that the bounds that we obtain depend only on $b_1(G)$. Removing the $1$-formality assumption, it is easy to construct examples where the interval where $\theta_q(G)\neq 0$ is arbitrarily large relative to $b_1(G)$ (see Example~\ref{ex=freenilp}). We remark here that the assumption $b_1(G)\geq 3$ is made in order to make the conclusion of Theorem~\ref{thm:Kahler} uniform: if $b_1(G)\leq 2$ then $\theta_q(G)=0$ for all $q\geq 2$ (but not for $q=1$). As in Theorem~\ref{thm:Kahler}, all the subsequent estimates and vanishing results for $\theta_q(G)$ will have a borderline case $q=b_1(G)-1$. Since our theory is set up to control $\theta_q(G)$ for $q\geq 2$ (essentially due to \eqref{eq:theta-less-W} below), it will simplify the discussion to assume that $b_1(G)\geq 3$. Building on Theorem~\ref{thm:Kahler}, we prove the following in Section~\ref{subsec:deg-growth}.

\begin{thm}\label{thm:growth-degree}
 Let $G$ be a finitely generated $1$-formal group, and let $M = G/G''$ denote its maximal metabelian quotient. If $\Gamma_q(M)$ is finite for $q\gg 0$ (in particular if $M$ is nilpotent) and if $n=b_1(G)\geq 3$, then
 \begin{enumerate}
  \item\label{it:conc-1} $\Gamma_{n-1}(M)$ is finite.
  \item\label{it:conc-2} $M$ contains a nilpotent finite index subgroup with nilpotency class at most $n-2$.
  \item\label{it:conc-3} The group $M$ has polynomial growth of degree $d(M)$ bounded above by
 \[d(M) \leq n + (n-2)\cdot{2n-3\choose n-4}.\]
 \end{enumerate}
\end{thm}

We recall that the \emph{nilpotency class} of a nilpotent group $H$ is the largest value of $c$ for which $\Gamma_c(H)\neq \{1\}$, and that a group is \emph{virtually nilpotent} if it contains a nilpotent finite index subgroup. We can then define the \defi{virtual nilpotency class (vnc)} of a group to be the minimal nilpotency class of a finite index subgroup, so that the second conclusion of Theorem~\ref{thm:growth-degree} states that the metabelian quotient $G/G''$ has $\vnc(G/G'')\leq n-2$. A celebrated theorem of Gromov \cite{gromov} asserts that virtually nilpotent groups are precisely the ones that exhibit polynomial growth, and our theorem gives a precise bound for the degree of growth of $M=G/G''$ solely in terms of the first Betti number of $G$. As we show in Section~\ref{subsec:deg-growth}, this bound is a consequence of the Bass--Guivarc'h formula and the estimate on $\theta_q(G)$ as in Theorem~\ref{thm:Kahler}. As explained in Example~\ref{ex=freenilp} and Remark~\ref{rem:freenilp}, in the absence of $1$-formality (or some appropriate replacement) there can be no such bound for either $\vnc(M)$ or for $d(M)$.

\vskip 3pt

An important special case of Theorem~\ref{thm:growth-degree} concerns the Torelli group $T_g$ which measures from a homotopical point of view the difference between the moduli space of curves of genus $g$ and that of principally polarized abelian varieties of dimension $g$. If $\mathrm{Mod}_g$ denotes the mapping class group, recall that the Torelli group is defined via the exact sequence
$$1\longrightarrow T_g\longrightarrow \mathrm{Mod}_g\longrightarrow \mbox{Sp}_{2g}(\mathbb Z)\longrightarrow 1.$$
Using Johnson's fundamental calculation $H^1(T_g,\mathbb Q)\cong \bigwedge ^3 H_{\mathbb Q}/H_{\mathbb Q}$, where $H$ is the first integral homology of a genus $g$ Riemann surface (and hence $\mbox{rk} (H)=2g$), together with the recent result by Ershov and He \cite[Corollary~1.5]{EH} asserting that $T_g/T_g''$ is nilpotent for $g\geq 12$, and the $1$-formality of $T_g$ proved in \cite{Hai}, our Theorem~\ref{thm:growth-degree} yields the following consequence (see Section~\ref{subsec:torelli}).

\begin{thm}\label{thmB}
If $g\geq 12$,  the metabelian quotient $T_g/T_g''$ has virtual nilpotency class at most $${2g\choose 3}-2g-2.$$
\end{thm}

\noindent We note that some of the results of \cite{EH} have been extended to $g\geq 4$ in \cite{CEP}, but the nilpotence of $T_g/T_g''$ is still open, preventing us from extending the range for $g$ in Theorem~\ref{thmB}.

We have similar applications  to the Torelli group $\OA_g$ of the free group $F_g$ on $g$ generators, discussed in Section~\ref{subsec:torelli}. We show in Theorem~\ref{thm=oan} that
\[\vnc(\OA_g/\OA_g'') \leq \frac{g(g+1)(g-2)}{2}-2\quad\mbox{ for all }g\geq 4,\]
which is an analogue of Theorem~\ref{thmB}. We note that in contrast with the Torelli groups~$T_g$, the groups $\OA_g$ are not known to be $1$-formal \cite[Question~10.6]{PS-johnson}.

\vskip 3pt

The proofs of Theorems~\ref{thm:Kahler} and~\ref{thm:growth-degree} are fundamentally algebraic. In order to make the transition to commutative algebra and algebraic geometry, for every finitely generated group $G$ we consider its \emph{Alexander invariant}
$$B(G):=H_1(G',\mathbb Z)=G'/G'',$$
viewed as a module over the group ring $\mathbb Z[G/G']$, with the action being given by conjugation. Using the \emph{augmentation ideal}
$I\subseteq \mathbb Z[G/G']$, one then constructs the graded module $\mbox{gr}\,B(G)$ over the ring $\mbox{gr } \mathbb Z[G/G']=\bigoplus_{q\geq 0} I^q/I^{q+1}$, by setting
$$\mbox{gr}_q B(G):=I^q\cdot B(G)/I^{q+1}\cdot B(G).$$
Massey proves in \cite{Massey} that $\mbox{gr}_{q+2}(G/G'')\cong I^q\cdot B(G)/I^{q+1}\cdot B(G)$, which in particular yields
$$\theta_{q+2}(G)=\mbox{rk } \mbox{gr}_q B(G).$$
Following \cite{PS-imrn} and especially \cite{PS15}, it turns out that the graded module $\mbox{gr}\,B(G)$ can be studied via Koszul cohomology using a purely algebraic construction which we describe next.

\vskip 3pt

\noindent {\bf Koszul modules.} Suppose that $V$ is an $n$-dimensional complex vector space and fix an $m$-dimensional subspace $K\subseteq \bigwedge^2 V$. We denote by $K^\perp=(\bigwedge^2V/K)^\vee\subseteq \bigwedge^2V^\vee$ the orthogonal complement of $K$.
We also denote by $S:=\mbox{Sym}(V)$ the polynomial algebra over $V$ and consider the Koszul complex
$$\cdots \longrightarrow \bigwedge^3 V\otimes S\stackrel{\delta_3}\longrightarrow \bigwedge^2 V\otimes S\stackrel{\delta_2}\longrightarrow V\otimes S\stackrel{\delta_1}\longrightarrow S\longrightarrow \mathbb C\longrightarrow 0.$$ According to \cite{PS15}, the  {\em Koszul module} associated to $(V,K)$ is the graded $S$-module
\[W(V,K):=\mathrm{Coker}\Bigl\{\bigwedge^3V\otimes S\longrightarrow \Bigl(\bigwedge^2V/K\Bigr)\otimes S\Bigr\},\]
where the map in question is the projection $\bw^2 V \oo S \to (\bigwedge^2V/K)\otimes S$ composed with the Koszul differential $\delta_3$. Our grading conventions are so that $W(V,K)$ is generated in degree $0$. Hence,
passing to individual graded pieces one has the identification $$W_q(V,K)\cong \bigwedge^2 V\otimes \mbox{Sym}^q(V)/\Bigl(\mbox{Ker}(\delta_{2,q})+K\otimes \mbox{Sym}^q(V)\Bigr).$$

\vskip 3pt

If $G$ is a finitely generated group, we define its \emph{Koszul module} $$W(G):=W\bigl(H_1(G,\mathbb C), K\bigr),$$ by letting $K = \partial_G\bigl(H_2(G,\mathbb C)\bigr)$, where $\partial_G$ is the dual of $\cup_G\colon \bigwedge^2 H^1(G,\mathbb C)\rightarrow H^2(G,\mathbb C)$, the cup product map on $G$. Equivalently, we have that $K^{\perp} = \mbox{Ker}\bigl(\cup_G\bigr)$. The origin of this purely algebraic definition of Koszul modules is \cite{PS-imrn}, where they are indirectly introduced in order to provide an explicit realization of the \emph{infinitesimal Alexander invariant} of $G$ as a graded module over the symmetric algebra $\mbox{Sym } H_1(G, \mathbb C)$.

It is shown in  \cite{PS-imrn} and \cite{MaP} that the following inequality holds
\begin{equation}\label{eq:theta-less-W}
\theta_{q+2}(G)=\mbox{rk } \mbox{gr}_q B(G)\leq \mbox{dim } W_q(G),
\end{equation}
with equality if the group $G$ is \emph{1-formal}. We refer to Proposition \ref{comparison} for a precise statement and further details in this direction. We prove the following result about~$W(G)$.

\begin{thm}\label{thmA}
Let $G$ be a finitely generated group and assume that the cup product map $$\cup_G\colon \bigwedge^2 H^1(G,\mathbb C)\rightarrow H^2(G,\mathbb C)$$ does not vanish on any non-zero decomposable elements. If $n=b_1(G)\geq 3$, then
\begin{equation}\label{eq:vanishing-WG}
W_q(G)=0 \ \ \mbox{ for all } \ \ q\geq n-3.
\end{equation}
If moreover $M=G/G''$ satisfies the condition that $\Gamma_q(M)$ be finite for $q\gg 0$, then its virtual nilpotency class and degree of polynomial growth satisfy the inequalities in Theorem~\ref{thm:growth-degree}.
\end{thm}


\vskip 3pt

The vanishing (\ref{eq:vanishing-WG}) follows from Theorem~\ref{thm:WG-RG}, while the inequalities for $d(M)$ and $\vnc(M)$ are explained in the proof of Theorem~\ref{thm:growth-degree} (see Remark~\ref{rem:vanishing-resonance}). An important application of Theorem \ref{thmA} is when $X$ is a compact K\"ahler manifold and $G=\pi_1(X)$. In this case $H^1(G,\mathbb C)=H^1(X,\mathbb C)$, and the kernel of $\cup_G$ equals the kernel of the map
$$\cup_X\colon \bigwedge^2 H^1(X,\mathbb C)\rightarrow H^2(X,\mathbb C).$$
By a well-known generalization of the Castelnuovo--de Franchis Theorem \cite{Cat}, the condition that $\cup_X$ does vanish on decomposable elements amounts to the existence of a fibration $X\rightarrow C$ over a smooth curve of genus $g\geq 2$. Thus, the first part of Theorem~\ref{thmA} applies to all compact K\"ahler manifolds $X$ with irregularity $q(X)\geq 3$ which are not fibred over curves.

\vskip 3pt

We obtain the vanishing in Theorem~\ref{thmA} as a special instance of a general vanishing result for Koszul modules, as explained next. We return to the general situation when $V$ is a vector space, $\iota\colon  K\hookrightarrow \bigwedge^2 V$ is a fixed subspace, and $K^\perp:=\ker(\iota^{\vee})\subseteq \bigwedge^2 V^{\vee}$ is the space of skew-symmetric bilinear forms on $V$ which vanish identically on $K$. It is shown in \cite[Lemma 2.4]{PS15} that the support of the Koszul module $W(V,K)$ in the affine space $V^\vee$ coincides (if non-empty) with the \emph{resonance variety} $\mc{R}(V,K)$ defined as
$$\mc{R}(V,K):=\Bigl\{a\in V^\vee  \, | \, \mbox{ there exists }b\in V^\vee \mbox{ such that } a\wedge b\in K^\perp\setminus \{0\} \Bigr\}\cup \{0\}.$$
In particular, $W(V,K)$ has finite length, that is, $W_q(V,K)=0$ for $q\gg 0$ if and only $\mc{R}(V,K)=\{0\}$. This last condition is equivalent to the fact that the projective subspace $\bb{P}(K^{\perp})\subseteq\bb{P}(\bw^2 V^{\vee})$ is disjoint from the Grassmann variety $\Gr_2(V^{\vee})$ in its Pl\"ucker embedding. Our vanishing result applies uniformly to all subspaces $K$ satisfying this condition, as follows (see Theorem~\ref{thm=main}).

\begin{thm-main*}
\emph{Let $V$ be a complex $n$-dimensional vector space and let $K\subseteq\bw^2 V$ be a subspace such that $\mc{R}(V,K)=\{0\}$. We have that $W_{q}(V,K)=0$ for all $q\geq n-3$}.
\end{thm-main*}

In the paper \cite{AFPRW}, we explain how via a degeneration argument, the Main Theorem gives rise to an essentially elementary proof of  Green's Vanishing Conjecture for syzygies of generic canonical curves. This leads to an alternative approach to Green's conjecture from the one of Voisin's \cite{V02}, \cite{V05}.

The vanishing result in the Main Theorem is responsible for the estimates of the virtual nilpotency classes that appear in Theorems~\ref{thm:growth-degree},~\ref{thmB} and~\ref{thmA}. In addition to the vanishing result, we also provide a sharp upper bound for the Hilbert function of a finite length Koszul module in Theorem~\ref{thm:HilbSeries}. In combination with (\ref{eq:theta-less-W}) this gives the bound on the Chen ranks in Theorem~\ref{thm:Kahler}, and in conjunction with the Bass--Guivarc'h formula (recalled as formula (\ref{thm:bass-guivarch})) it provides the estimate for the degree of polynomial growth in Theorem~\ref{thm:growth-degree}.

To emphasize the robust relationship between algebra and topology in the $1$-formal setting, we end our Introduction summarizing the connection between some of the invariants discussed in this setting.

\begin{thm}\label{thm:equivalences}
 Let $G$ be a finitely generated $1$-formal group, and assume that $n=b_1(G)\geq 3$. The following statements are equivalent:
 \begin{enumerate}
  \item[(a)] $\theta_q(G)=0$ for some (any) $q\gg 0$.
  \item[(b)]  $\theta_{n-1}(G)=0$.
  \item[(c)]  $\cup_G$ does not vanish on any non-zero decomposable elements.
  \item[(d)]  $W_q(G)=0$ for some (any) $q\gg 0$.
  \item[(e)]  $W_{n-3}(G)=0$.
 \end{enumerate}
\end{thm}

The equivalence of (c), (d), (e) comes from the Main Theorem, since all three conditions characterize finite dimensional Koszul modules. The equivalence of these statements with (a) and (b) follows from the fact that (\ref{eq:theta-less-W}) is an equality for $1$-formal groups.


\vskip 3pt

\begin{ack}
Above all, we acknowledge with thanks the contribution of Alex Suciu. This project started with the paper \cite{PS15}, and since then we benefited
from numerous discussions with him. We also thank A. Beauville, F. Campana, F. Catanese, D. Eisenbud, B. Farb, B. Klingler, P. Pirola, A. Putman and C. Voisin for interesting discussions related to this circle of ideas. We are grateful to the referees for many useful suggestions that significantly improved the presentation of this paper.

\vskip 3pt

{\small{Aprodu was partially supported by the Romanian Ministry of Research and
Innovation, CNCS - UEFISCDI, grant
PN-III-P4-ID-PCE-2016-0030, within PNCDI III. Farkas was supported by the DFG Grant \emph{Syzygien und Moduli} and by the ERC Advanced Grant SYZYGY of  the European Research Council (ERC) under the European Union Horizon 2020 research and innovation program (grant agreement No. 834172).
 Raicu was supported by the Alfred P. Sloan Foundation and by the NSF Grant No.~1901886. Weyman was partially supported by the Sidney Professorial Fund and the NSF grant No.~1400740.
}}
\end{ack}

\section{Koszul modules}
\label{sec:kmod}

Throughout this article, we denote by $V$ a complex vector space of dimension $n\ge 2$, we write $V^{\vee}=\Hom_{\bb{C}}(V,\bb{C})$ for its dual, and $S=\Sym V$ for the symmetric algebra of~$V$. We consider the standard grading on $S$ where the elements in $V$ are of degree 1. We fix a subspace $K\subseteq \bigwedge^2V$ of dimension~$m$, we denote by $\iota\colon K\to \bw^2V$ the inclusion and let $K^\perp:=\ker(\iota^{\vee}) =\coker(\iota)^\vee$, where $\iota^{\vee}\colon \bw^2 V^{\vee} \lra K^{\vee}$ is the dual of $\iota$.

\subsection{Definitions and basic properties}
\label{subsec:defs}

We recall from \cite{PS15} the definition of Koszul modules. They are defined starting from the classical Koszul differentials
\[
\delta_p\colon \bigwedge^pV\otimes S\to \bigwedge^{p-1}V\otimes S,
\]
\[
\delta_p(v_1 \wedge \cdots \wedge v_p \oo f) = \sum_{j=1}^p (-1)^{j-1} v_1 \wedge \cdots \wedge \widehat{v_j} \wedge \cdots \wedge v_p \oo v_j f.
\]
Note that we have a decomposition $\delta_p=\bigoplus_q\delta_{p,q}$, where
\begin{equation}\label{eq:delta-pq}
\delta_{p,q}\colon \bw^p V \oo \Sym^q V \to \bw^{p-1}V \oo \Sym^{q+1}V,
\end{equation}
and that the differentials $\delta_p$ fit together into the Koszul complex which gives a minimal resolution by free graded $S$-modules of the residue field $\mathbb C$.

\medskip

\begin{defn}[{\cite[\S2]{PS15}}]
The {\em Koszul module} associated to the pair $(V,K)$ is a graded $S$-module denoted by $W(V,K)$,  and defined as
\begin{equation}\label{eq:def-Koszul-module}
W(V,K):=\mathrm{Coker}\Bigl\{\bigwedge^3V\otimes S\lra \Bigl(\bigwedge^2V/K\Bigr)\otimes S\Bigr\}
\end{equation}
where the map $\bigwedge^3V\otimes S\to (\bigwedge^2V/K)\otimes S$ is the composition of the quotient map $\bw^2 V \oo S \to (\bigwedge^2V/K)\otimes S$ with the Koszul differential $\delta_3$. The grading is inherited from the symmetric algebra if we make the convention that $\bigwedge^2V/K$ is placed in degree~$0$ and $\bw^3 V$ is in degree~$1$. The degree $q$ component of $W(V,K)$ is then given by
\begin{equation}
\label{eqnWqDef}
W_q(V,K)=\mathrm{Coker}\Bigl\{\bigwedge^3V\otimes \Sym^{q-1}V\lra \Bigl(\bigwedge^2V/K\Bigr)\otimes \Sym^qV\Bigr\}.
\end{equation}
Note that this grading convention guarantees that $W(V,K)$ is generated in degree $0$. In particular, if $W_q(V,K)=0$ for some $q\geq 0$, then $W_p(V,K)=0$ for all $p \ge q$.
\end{defn}

Since the Koszul complex is exact in homological degree one, we can realize $W(V,K)$ as the middle cohomology of the complex
\begin{equation}\label{eqn:W}
\xymatrixcolsep{5pc}
\xymatrix{
K \oo S \ar[r]^{\delta_2|_{K \oo S}} & V\oo S \ar[r]^{\delta_1} & S
}
\end{equation}
Using the fact that $\ker(\delta_{1,q+1}) = \im(\delta_{2,q})$, we obtain as in \cite[(2.1)]{PS15}
\begin{equation}
\label{eqn:Wq_2}
W_q(V,K)=\mathrm{Im}(\delta_{2,q})/\mathrm{Im}(\delta_{2,q}|_{K\oo S}).
\end{equation}

This construction is natural in the following sense: for $K \subseteq K'$, the identity map of $\bigwedge^2V\otimes S$
induces a surjective morphism of graded $S$-modules
\begin{equation}
\label{eq:wnat}
W(V,K) \twoheadrightarrow W(V,K').
\end{equation}

\begin{ex}\label{ex:extremal-Koszul}
In the extremal cases the Koszul modules are easy to compute:
\begin{itemize}
\item If $K=\bigwedge^2V$ then $W(V,K)=0$;
\item If $K=0$ then $W_q(V,K)=H^0\bigl(\mathbb{P}(V^\vee ),\Omega^1_{\mathbb P(V^{\vee})} (q+2)\bigr)\neq 0$ for all $q\geq 0$.
\end{itemize}
We will therefore assume when necessary that $1\leq m=\mbox{dim}(K) < {n\choose 2}$.
\end{ex}

\begin{ex}[{\cite[\S 6.2]{PS15}}]\label{ex:W(G)}
Any finitely generated group $G$ gives rise to a pair $(V,K)$ by letting $V:=H_1(G,\mathbb C)\cong  H^1(G,\mathbb C)^\vee$ and $K$ be the the dual of the image of the cup--product map $\cup_G\colon \bigwedge^2H^1(G,\mathbb C)\to H^2(G,\mathbb C)$. The corresponding Koszul modules are denoted $W(G)$. Conversely, it was noted in \cite[Proposition~6.2]{PS15} that every Koszul module $W(V,K)$ corresponding to a rational subspace $K\subseteq\bigwedge^2V$ has the form $W(G)$ for some finitely presented group $G$. A more thorough analysis of the Koszul modules associated to groups will be discussed in \S \ref{subsec:top}.
\end{ex}

\subsection{Resonance}
\label{subsec:resonance}

Inspired by work of Green and Lazarsfeld in~\cite{GL91}, Papadima and Suciu  defined in \cite{PS15} resonance varieties in a purely algebraic context as follows.

\begin{defn}\label{def:resonance}
The {\em resonance variety} associated  to the pair $(V,K)$ is the locus
\begin{equation}
\label{eq:defr}
\mathcal R(V,K):=\left\{a\in V^\vee  \, : \, \mbox{ there exists }b\in V^\vee \mbox{ such that } a\wedge b\in K^\perp\setminus \{0\} \right\}\cup \{0\}
\end{equation}
\end{defn}

Geometrically, the resonance variety $\mathcal{R}(V,K)$ is the union of $2$-dimensional subspaces of $V^{\vee}$ parametrized by the intersection $\mathbb{P}(K^\perp)\cap\mathrm{Gr}_2(V^\vee )$. More precisely, via the canonical diagram
\[
\xymatrixcolsep{5pc}
\xymatrix{
\left(\mathbb{P}V^\vee\times\mathbb{P}V^\vee\right)\setminus\mathrm{Diag}\ar[r]^{\pi} \ar[d]^{p_1}& \mathrm{Gr}_2(V^\vee )\\
\mathbb{P}V^\vee &
}
\]
the variety $\mathcal{R}(V,K)$ is the affine cone over $p_1(\pi^{-1}(\mathbb P(K^\perp)\cap\mathrm{Gr}_2(V^\vee )))$. It was shown in \cite[Lemma~2.4]{PS15} that resonance and Koszul modules are strongly related: the support of the Koszul module in the affine space $V^\vee $ coincides with the resonance variety away from $0$, and hence $\mathcal R(V,K)=\{0\}$ if and only if $W(V,K)$ is of finite length. In particular
\begin{equation}
\label{eq:mainequiv}
\mathbb P(K^\perp)\cap\mathrm{Gr}_2(V^\vee )=\emptyset \Longleftrightarrow \mathcal R(V,K)= \{0\}
\Longleftrightarrow \dim_{\mathbb C}W(V,K) < \infty.
\end{equation}

Recalling that $m=\mbox{dim}(K)$, it was also proved in \cite[Proposition~2.10]{PS15} that there exists a uniform bound $q(n,m)$ such that $W_q(V,K)=0$ for all $q\ge q(n,m)$
and for all $(V,K)$ satisfying  $\mathbb P(K^\perp)\cap\mathrm{Gr}_2(V^\vee )=\emptyset$.
A central goal of our paper is to determine this bound, which we do in Theorem~\ref{thm=main}. For dimension reasons, if $\mathbb P(K^\perp)\cap\mathrm{Gr}_2(V^\vee )=\emptyset$ then $m\ge 2n-3$. The
{\em borderline case} $m = 2n-3$ is therefore of maximal interest, as illustrated for instance in \S\ref{subsec:HilbSeries} below.

\

\section{Finite length Koszul modules}
\label{sec:vanish}

The goal of this section is to prove the main theorem of our paper, characterizing finite dimensional Koszul modules (those for which the equivalent conditions in (\ref{eq:mainequiv}) hold) in terms of the vanishing of a fixed graded component. More precisely, we prove:

\begin{thm}\label{thm=main}
Let $V$ be a vector space of dimension $n\geq 3$ and let $K\subseteq\bw^2 V$ be an arbitrary subspace. We have that
\begin{equation}\label{eq:thm-main-equiv}
\mc{R}(V,K) = \{0\} \Longleftrightarrow W_{n-3}(V,K) = 0.
\end{equation}
\end{thm}

The assumption that $n\geq 3$ in the above theorem is made in order to avoid trivialities: when $n=1$ we have $\bw^2 V=0$, while for $n=2$ we have that either $K=\bw^2 V$ or $K=0$, both of which have been discussed in Example~\ref{ex:extremal-Koszul}. In addition to the vanishing result above, in the borderline case $\mbox{dim}(K)=2n-3$ we give an exact description of the Hilbert series of the Koszul modules with vanishing resonance, which in turn provides an upper bound for such Hilbert series in the general case $\mbox{dim}(K)\geq 2n-3$, as follows.

\vskip 3pt

\begin{thm}
\label{thm:HilbSeries}
Let $V$ be a vector space of dimension $n\geq 3$ and let $K\subseteq\bw^2 V$ be an arbitrary subspace. If $\mathcal{R}(V,K) = \{0\}$, then for $q=0, \ldots, n-4$, we have that
\[
\dim\, W_q(V,K) \leq {n+q-1 \choose q} \frac{(n-2)(n-q-3)}{q+2},
\]
and equality holds for all $0\leq q\leq n-4$ when $\dim(K)=2n-3$.
\end{thm}

We prove Theorem~\ref{thm=main} using Bott's theorem and a hypercohomology spectral sequence that was featured in the work of Voisin on the Green conjecture \cite{V02}. We recall the special case of Bott's theorem needed for our argument in Section~\ref{subsec:bott}, and prove the vanishing theorem in Section~\ref{subsec:bordervanish}. The proof of Theorem~\ref{thm:HilbSeries} is presented in Section~\ref{subsec:HilbSeries}, and is summarized as follows. We reinterpret the characterization (\ref{eq:thm-main-equiv}) in terms of a transversality condition on $K$ with respect to the kernel of the Koszul differential $\delta_2$, which allows us to prove the conclusion of Theorem~\ref{thm:HilbSeries} in the borderline case. The general case of Theorem~\ref{thm:HilbSeries} reduces to the borderline case via a generic projection argument.

\subsection{Bott's Theorem for Grassmannians}\label{subsec:bott} Our main reference for this topic is \cite[Ch.~4]{weyman}.
Let $\GG=\Gr_2(V^{\vee})$ denote the Grassmannian of $2$--dimensional subspaces of $V^{\vee}$, which we think of by duality as parametrizing $2$-dimensional quotients of $V$. On $\GG$ we then have the tautological exact sequence
\begin{equation}\label{eq:tautological}
0\lra\cU \lra V\oo\mc{O}_{\GG}\lra\mc{Q}\lra 0,
\end{equation}
where $\cU$ (respectively $\mc{Q}$) denotes the universal rank $(n-2)$ sub--bundle (respectively rank $2$ quotient bundle) of the trivial bundle $V$. We write $\mc{O}_{\GG}(1)$ for $\bw^2\mc{Q}$, which is the Pl\"ucker line bundle giving an embedding of $\GG$ into $\bb{P}(\bw^2 V^{\vee})$ as the projectivization of the set of decomposable $2$-forms $a\wedge b$, with $a,b\in V^{\vee}$.

\vskip 4pt

We write $\bb{Z}^r_{\op{dom}}$ for the set of \emph{dominant weights} in $\bb{Z}^r$, that is, tuples $\nu=(\nu_1,\ldots,\nu_r)\in\bb{Z}^r$ with $\nu_1\geq\nu_2\geq\ldots\geq\nu_r$. We write $\bb{S}_{\nu}$ for the \emph{Schur functor} associated with $\nu$, and recall that if $\nu=(a,0,\ldots,0)$ then $\bb{S}_{\nu} = \Sym^a$, and if $\nu=(1,1,\ldots,1)$ then $\bb{S}_{\nu} = \bw^r$. A special case that will be of interest to us is when $\a=(\a_1,\a_2)\in\bb{Z}^2_{\op{dom}}$, in which case we have
\begin{equation}\label{eq:SS-alpha}
\bb{S}_{\a}\mc{Q} = \Sym^{\a_1-\a_2}\mc{Q} \,\oo \mc{O}_{\GG}(\a_2).
\end{equation}
The cohomology groups of sheaves of the form $\bb{S}_{\a}\mc{Q}\oo\bb{S}_{\b}\cU$ are completely described by Bott's Theorem \cite[Corollary~4.1.9]{weyman}. We will only need a special case of this theorem, which we record next.

\begin{thm}[Bott]\label{thm:bott}
 Let $\a\in\bb{Z}^2_{\op{dom}}$ and  $\b\in\bb{Z}^{n-2}_{\op{dom}}$, and let $\gamma=(\a|\b)\in\bb{Z}^n$ denote their concatenation. If $\gamma$ is dominant then
 \begin{equation}\label{eq:gamma-dominant}
 H^0(\GG,\bb{S}_{\a}\mc{Q}\oo\bb{S}_{\beta}\cU) = \bb{S}_{\gamma}V,\mbox{ and }H^j(\GG,\bb{S}_{\a}\mc{Q}\oo\bb{S}_{\beta}\cU)=0\mbox{ for all }j>0.
 \end{equation}
 If there exist $1\leq x<y\leq n$ with $\gamma_x-x=\gamma_y-y$, then
\begin{equation}\label{eq:Bott-vanishing}
H^j(\GG,\bb{S}_{\a}\mc{Q}\oo\bb{S}_{\beta}\cU)=0,\mbox{ for all }j\geq 0.
\end{equation}
In particular,
\begin{enumerate}
\item\label{it:1} If $\a_1\geq\a_2\geq 0$, then $H^0(\GG,\bb{S}_{\a}\mc{Q}) = \bb{S}_{\a}V$ and $H^j(\GG,\bb{S}_{\a}\mc{Q})=0$ for all $j>0$.
\item\label{it:2} For all $a,j\geq 0$ we have that $H^j(\GG,\Sym^a\mc{Q}\oo\cU) = 0$.
\item\label{it:3} If $-2\geq\a_1\geq 1-n$, or $-1\geq\a_2\geq 2-n$, then $H^j(\GG,\bb{S}_{\a}\mc{Q})=0$ for all $j\geq 0$.
\end{enumerate}
\end{thm}

To see how (\ref{it:1}) and (\ref{it:3}) follow from the general statement of the theorem we specialize to $\b=(0,0,\ldots,0)$, so that $\gamma=(\a_1,\a_2,0,\ldots,0)$. In case (\ref{it:1}) we have that $\gamma$ is dominant so we can apply (\ref{eq:gamma-dominant}), while in case (\ref{it:3}) we apply (\ref{eq:Bott-vanishing}) with $(x,y) = (1,1-\a_1)$ when $-2\geq\a_1\geq 1-n$, and with $(x,y) = (2,2-\a_2)$ when $-1\geq\a_2\geq 2-n$. To see how (\ref{it:2}) follows from (\ref{eq:Bott-vanishing}) we note that $\a=(a,0)$, $\b=(1,0,\ldots,0)$, and take $(x,y) = (2,3)$.


\subsection{Vanishing for Koszul modules}
\label{subsec:bordervanish}

We begin by giving a geometric construction of $\ker(\delta_1)$, where $\delta_1\colon V\oo S\lra S$ is the first Koszul differential. We let $\GG=\Gr_2(V^{\vee})$ as in the previous section, and consider
\[\mc{S} = \Sym_{\mc{O}_{\GG}}(\mc{Q}) = \mc{O}_{\GG} \oplus \mc{Q} \oplus \Sym^2\mc{Q} \oplus \cdots\]
as a sheaf of graded $\mc{O}_{\GG}$-algebras on $\GG$. Locally on $\GG$, the sheaf $\mc{S}$ can be identified with a polynomial ring in two variables, where $\mc{Q}$ is the space of linear forms. We can then define the Koszul differentials
\[\delta_1^{\mc{Q}} \colon \mc{Q} \oo_{\mc{O}_{\GG}} \mc{S} \lra \mc{S}\quad\mbox{ and }\quad\delta_2^{\mc{Q}} \colon \bw^2\mc{Q} \oo_{\mc{O}_{\GG}} \mc{S} \lra \mc{Q} \oo_{\mc{O}_{\GG}} \mc{S}\]
as in Section~\ref{subsec:defs}. Recalling that $\bw^2\mc{Q}=\mc{O}_{\GG}(1)$, we write $\bw^2\mc{Q} \oo_{\mc{O}_{\GG}} \mc{S}$ simply as $\mc{S}(1)$.

\begin{lem}\label{lem:H0-mcS} We have a commutative diagram of graded $S$--modules
\begin{equation}\label{eq:H0-mcS}
\begin{gathered}
\xymatrix{
& \bw^2 V \oo S \ar[rr]^{\delta_2} \ar[d]_{f_2} & & V \oo S \ar[rrr]^{\delta_1} \ar[d]_{f_1} & & & S \ar[d]_{f_0} \\
0\ar[r] & H^0(\GG,\mc{S}(1)) \ar[rr]^{H^0(\GG,\delta_2^{\mc{Q}})} & & H^0(\GG,\mc{Q} \oo \mc{S}) \ar[rrr]^{H^0(\GG,\delta_1^{\mc{Q}})} & & & H^0(\GG,\mc{S}) \\
}
\end{gathered}
\end{equation}
where the maps $f_1$ and $f_0$ are isomorphisms. Moreover, we have a natural isomorphism between $\ker(\delta_1)$ and $H^0(\GG,\mc{S}(1))$.
\end{lem}

 \begin{proof}
 If we identify $S\oo\mc{O}_{\GG}$ with $\Sym_{\mc{O}_{\GG}}(V\oo\mc{O}_{\GG})$ then the tautological quotient map $V\oo\mc{O}_{\GG}\onto\mc{Q}$ in (\ref{eq:tautological}) yields a natural surjection
 \[S\oo\mc{O}_{\GG} \onto \mc{S}\]
 which can be thought of locally on $\GG$ as realizing the polynomial ring $\mc{S}$ in two variables as a quotient of $S$ by the ideal generated by the linear forms in $\cU$. The naturality of the Koszul complex yields a commutative diagram
 \[
 \xymatrix{
& \bw^2 V \oo S \oo \mc{O}_{\GG} \ar[r] \ar[d] & V \oo S \oo \mc{O}_{\GG} \ar[r] \ar[d] & S \oo \mc{O}_{\GG} \ar[d] \\
0 \ar[r] & \bw^2 \mc{Q} \oo \mc{S} \ar[r] & \mc{Q} \oo \mc{S} \ar[r] & \mc{S} \\
 }
 \]
 where the bottom row is the full Koszul complex since $\mc{Q}$ has rank $2$. The diagram (\ref{eq:H0-mcS}) is then obtained by applying the global sections functor $H^0(\GG,\bullet)$. It is sufficient to prove that $f_0$ and $f_1$ are isomorphisms, since this implies that $f_1$ induces an isomorphism between $\ker(\delta_1)$ and $\ker(H^0(\GG,\delta_1^{\mc{Q}}))$, and the latter is in turn is isomorphic to $H^0(\GG,\mc{S}(1))$ via the map $H^0(\GG,\delta_2^{\mc{Q}})$.

\vskip 4pt

To prove that $f_0$ is an isomorphism it suffices to prove it is injective, since for each~$q$ we have that $H^0(\GG,\Sym^q\mc{Q}) = \Sym^q V$ by conclusion (\ref{it:1}) of Theorem~\ref{thm:bott}. If we let $Y = \ul{\rm{Spec}}_{\GG}(\mc{S})$ denote the total space of the bundle $\mc{Q}^{\vee}$ (see \cite[Exercise~II.5.18]{hartshorne}), then $Y$ is a subbundle of the trivial bundle $V^{\vee}$ over $\GG$, which can be described explicitly as
\[ Y = \Bigl\{ (c,[a\wedge b]) \in V^{\vee} \times \GG : c\in\mbox{Span}(a,b)\Bigr\}\]
If we write $\pi\colon Y\to V^{\vee}$ for the natural projection map, then $\pi$ is surjective and therefore the induced pull-back map $f_0\colon S\lra H^0(V^{\vee},\pi_*\mc{O}_Y)=H^0(\GG,\mc{S})$ is injective, as desired.

To prove that $f_1$ is an isomorphism, we observe that it factors as the composition of
\[V\oo S \overset{\op{id}_V\oo f_0}{\lra} V \oo H^0(\GG,\mc{S}) = H^0(\GG,V\oo\mc{S}),\]
which is an isomorphism, with the map $H^0(\GG,V\oo\mc{S})\lra H^0(\GG,\mc{Q}\oo\mc{S})$ induced by the tautological surjection $V\oo\mc{O}_{\GG}\onto\mc{Q}$. To prove that this map is an isomorphism, it is enough to check that
\[ H^0(\GG,\cU\oo\Sym^q\mc{Q}) = H^1(\GG,\cU \oo\Sym^q\mc{Q}) = 0\mbox{ for all }q\geq 0,\]
which follows from conclusion (\ref{it:2}) of Theorem~\ref{thm:bott}.
\end{proof}

We next recall that $W(V,K)$ can be thought of as the middle cohomology of (\ref{eqn:W}), or equivalently as
\[
\begin{aligned}
\xymatrixcolsep{5pc}
W(V,K) &= \coker\Bigl\{
\xymatrix{
K \oo S \ar[rr]^{\delta_2|_{K \oo S}} & & \ker(\delta_1)
}\Bigr\} \\
&\cong \coker\Bigl\{
\xymatrix{
K \oo S \ar[rr]^{f_2|_{K \oo S}} & & H^0(\GG,\mc{S}(1))
}\Bigr\} \\
\end{aligned}
\]
where the isomorphism above follows from Lemma~\ref{lem:H0-mcS}. Using again the isomorphism $f_0\colon S\lra H^0(\GG,\mc{S})$ and the fact that $K\oo H^0(\GG,\mc{S}) = H^0(\GG,K\oo\mc{S})$, we can rewrite
\begin{equation}\label{eq:WVK-from-cohom}
 W(V,K) = \coker\Bigl\{
\xymatrix{
H^0(\GG,K\oo\mc{S}) \ar[rr]^{H^0(\GG,\eta)} & & H^0(\GG,\mc{S}(1))
}\Bigr\}
\end{equation}
where
\begin{equation}\label{eq:def-eta}
\eta\colon K\oo\mc{S} \lra \mc{S}(1)\mbox{ is induced by }K\oo\mc{O}_{\GG}\hookrightarrow\bw^2 V \oo\mc{O}_{\GG} \onto \bw^2\mc{Q} = \mc{O}_{\GG}(1).
\end{equation}
Equipped with the description (\ref{eq:WVK-from-cohom}) of the Koszul modules, we can now give the proof of our main theorem.

\vskip 4pt

\begin{proof}[Proof of Theorem~\ref{thm=main}]
If $W_{n-3}(V,K)= 0$, then $\dim\, W(V,K)< \infty$, since the graded module $W(V,K)$ is generated in degree $0$.
It follows from \eqref{eq:mainequiv} that $\mathcal{R} (V,K)=\{ 0\}$.

\vskip 3pt

Conversely, suppose that $\mathcal{R} (V,K)=\{ 0\}$, which by (\ref{eq:mainequiv}) is equivalent to the fact that $\bb{P}(K^{\perp})$ is disjoint from the Grassmannian $\GG$, which in turn is equivalent to the fact that the linear system of divisors $(K,\mc{O}_{\GG}(1))$ is base point free, or to the fact that the map $K\oo\mc{O}_{\GG}\lra\mc{O}_{\GG}(1)$ is surjective. The Koszul complex associated to this map is therefore exact, and takes the form
\[\mc{K}^{\bullet}: 0\lra \bw^m K\oo \mc{O}_{\GG}(1-m) \lra \cdots \lra \bw^2 K\oo \mc{O}_{\GG}(-1) \lra K \oo \mc{O}_{\GG} \lra \mc{O}_{\GG}(1) \lra 0\]
where we use a cohomological grading with $\mc{K}^{-i} = \bw^i K \oo \mc{O}_{\GG}(1-i)$. Tensoring with $\Sym^{n-3}\mc{Q}$ preserves exactness, so the hypercohomology of $\mc{K}^{\bullet}\oo\Sym^{n-3}\mc{Q}$ vanishes identically. The hypercohomology spectral sequence associated with $\mc{K}^{\bullet}\oo\Sym^{n-3}\mc{Q}$ then yields that
 \[ E_1^{-i,j} = H^j\Bigl(\GG,\bw^i K\oo \mc{O}_{\GG}(1-i) \oo \Sym^{n-3}\mc{Q}\Bigr)  \Longrightarrow 0.\]
 Since $\Sym^{n-3}\mc{Q}\oo\mc{O}_{\GG}(1-i) = \bb{S}_{(n-2-i,1-i)}\mc{Q}$ by (\ref{eq:SS-alpha}), we can rewrite the terms in the spectral sequence as
 \begin{equation}\label{eq:Eij}
  E_1^{-i,j} = H^j\Bigl(\GG,\bw^i K \oo \bb{S}_{(n-2-i,1-i)}\mc{Q}\Bigr) = \bw^i K \oo H^j\bigl(\GG,\bb{S}_{(n-2-i,1-i)}\mc{Q}\bigr).
 \end{equation}
 We have using (\ref{eq:WVK-from-cohom}) the following identification
 \[W_{n-3}(V,K) = \coker\bigr\{E^{-1,0}_1 \lra E^{0,0}_1\bigr\}\]
 and we suppose by contradiction that this map is not surjective. Since $E_{\infty}^{0,0}=0$, there must be some \emph{non-zero} differential
 \[E^{-r,r-1}_r \lra E_r^{0,0}\mbox{ for }r\geq 2,\mbox{ or }E_r^{0,0} \lra E_r^{r,1-r}\mbox{ for }r\geq 1.\]
 Since $E_1^{r,1-r}=0$ for all $r\geq 1$ it follows that $E_r^{r,1-r}=0$ as well, so the latter case does not occur. To prove that the former case does not occur either and obtain a contradiction, it suffices to check that $E^{-r,r-1}_1=0$ for all $r\geq 2$, which follows from (\ref{eq:Eij}) if we can prove that
 \begin{equation}\label{eq:vanishing-Hr-1}
 H^{r-1}\bigl(\GG,\bb{S}_{(n-2-r,1-r)}\mc{Q}\bigr) = 0\mbox{ for }r\geq 2.
 \end{equation}
 If we write $\a=(n-2-r,1-r)$ then for $2\leq r\leq n-1$ we have that $-1\geq\a_2\geq 2-n$, and for $n\leq r\leq 2n-3$ we have that $-2\geq\a_1\geq 1-n$, so (\ref{eq:vanishing-Hr-1}) follows from conclusion (\ref{it:3}) of Theorem~\ref{thm:bott}. When $r>2n-3$ we get that $r-1>2n-4=\dim(\GG)$, so (\ref{eq:vanishing-Hr-1}) is a consequence of Grothendieck's vanishing theorem \cite[Theorem~III.2.7]{hartshorne}.
\end{proof}

%
%

\subsection{The Hilbert series of Koszul modules}\label{subsec:HilbSeries}
The goal of this section is to prove Theorem~\ref{thm:HilbSeries}. Recall from equation \eqref{eqn:Wq_2} that $W_q(V,K)$ is obtained as the cokernel of the map
\begin{equation}\label{eq:presentation-Wq}
\delta_{2,q}\colon K\otimes \Sym^q V \lra \Im(\delta_{2,q})
\end{equation}
and in particular we have that $W_q(V,K) \neq 0$ when $\dim(\Im(\delta_{2,q}))>\dim(K\otimes \Sym^qV)$. We say that the {\em divisorial case} occurs when the map (\ref{eq:presentation-Wq}) is given by a square matrix, i.e.
\[
\mathrm{dim}\, (K\otimes \Sym^qV)=\mathrm{dim}\left(\mathrm{Im}(\delta_{2,q})\right).
\]


\begin{prop}
\label{prop:nonzero}
Let $K\subseteq \bigwedge^2V$ be an arbitrary subspace of dimension $m=2n-3$. For any $0\le q\le n-4$ we have $W_q(V,K)\ne 0$, and $q=n-3$ is a divisorial case.
\end{prop}

\proof
Using the fact that $\dim\Sym^k V= {n+k-1 \choose k}$ and $\dim(K)=2n-3$ we obtain
\[
\mathrm{dim}\, (K\otimes \Sym^qV)=(2n-3){n+q-1 \choose q}.
\]
The exactness of the Koszul complex implies that $\delta_{1,q+1}$ is surjective and that we have the equality $\Im(\delta_{2,q}) = \ker(\delta_{1,q+1})$, so
\[
\mathrm{dim}(\Im(\delta_{2,q}))= \mathrm{dim}\, (V\otimes \Sym^{q+1}V) - \mathrm{dim}\, (\Sym^{q+2}V) = n\cdot{n+q\choose q+1} - {n+q+1\choose q+2},
\]
A direct computation leads for $q\ge 0$ to the formula
\begin{equation}
\label{eqn:dimW}
\mathrm{dim}(\Im(\delta_{2,q})) - \mathrm{dim}\, (K\otimes \Sym^qV)= {n+q-1 \choose q} \frac{(n-2)(n-q-3)}{q+2},
\end{equation}
which is positive for $0\leq q\leq n-4$ and vanishes for $q=n-3$, proving that $W_q(V,K) \neq 0$ for $0\leq q\leq n-4$, and that the case $q=n-3$ is divisorial.
\endproof

We next analyze in more detail the divisorial case, where the vanishing of $W_{n-3}(V,K)$ can be rephrased into a transversality condition as follows.

\begin{lem}
\label{lem:Schubert1}
Let $K\subseteq \bigwedge^2V$ be a subspace of dimension $2n-3$. The following are equivalent:
\begin{enumerate}
\item[(a)] $W_{n-3}(V,K) = 0$.
\item[(b)] $(K\otimes \Sym^{n-3}V)\cap \ker(\delta_{2,n-3})=0$.
\item[(c)] $(K\otimes \Sym^{q}V)\cap \ker(\delta_{2,q})=0$ for all $q\le n-3$.
\end{enumerate}
\end{lem}

\proof
It is clear that (c) implies (b), while for the reverse implication it suffices to observe that multiplication by a linear form in $V$ yields for each $q\geq 0$ an injective map
from $(K\otimes \Sym^{q}V)\cap \ker(\delta_{2,q})$ to $(K\otimes \Sym^{q+1}V)\cap \ker(\delta_{2,q+1})$.

Using the isomorphism $\im(\delta_{2,n-3}) = (\bw^2 V \oo \Sym^{n-3}V) / \ker(\delta_{2,n-3})$ and (\ref{eqn:Wq_2}) we get
\begin{equation}\label{eq:Wn-3=quot}
W_{n-3}(V,K)=\bigl(\bigwedge^2V\otimes \Sym^{n-3}V\bigr)/\bigl(K\otimes \Sym^{n-3}V+\ker(\delta_{2,n-3})\bigr).
\end{equation}
Since we are in the divisorial case we have $\dim(\im(\delta_{2,n-3}))=\dim(K\otimes \Sym^{n-3}V)$, thus
\begin{equation}\label{eq:dimensions-match}
\mathrm{dim}(\bigwedge^2V\otimes \Sym^{n-3}V)=\mathrm{dim}(K\otimes \Sym^{n-3}V)+\mathrm{dim}(\ker(\delta_{2,n-3})).
\end{equation}
Combining (\ref{eq:Wn-3=quot}) with (\ref{eq:dimensions-match}) we conclude that $W_{n-3}(V,K)=0$ if and only if the sum $K\otimes \Sym^{n-3}V+\ker(\delta_{2,n-3})$ is direct, proving the equivalence of (a) and (b).
\endproof

\begin{cor}\label{cor:maximal-rank}
If $K\subseteq\bigwedge^2V$ has dimension $2n-3$ and $\mathcal{R}(V,K) = \{0\}$, then the map (\ref{eq:presentation-Wq}) is of maximal rank for all $q$.
\end{cor}

\proof
Using (\ref{eqn:dimW}), it suffices to prove that (\ref{eq:presentation-Wq}) is injective for $q\leq n-3$ and surjective for $q\geq n-3$. Since $\mathcal{R}(V,K) = \{0\}$, we know from Theorem~\ref{thm=main} that $W_{n-3}(V,K)=0$, hence the assertion in Lemma~\ref{lem:Schubert1}(c) holds, proving the injectivity of (\ref{eq:presentation-Wq}) in the range $q\leq n-3$. The surjectivity of (\ref{eq:presentation-Wq}) is equivalent to the vanishing of $W_q(V,K)$, which is true for $q\geq n-3$ because $W_{n-3}(V,K)=0$ and $W(V,K)$ is generated in degree $0$.
\endproof

We now have all the tools necessary to prove Theorem~\ref{thm:HilbSeries}.

\begin{proof}[Proof of Theorem~\ref{thm:HilbSeries}]
We write $m=\dim(K)$ as usual, and suppose first that $m=2n-3$. By Corollary~\ref{cor:maximal-rank} we know that the maps (\ref{eq:presentation-Wq}) have maximal rank for all $q$. Since their cokernels are the graded components of $W(V,K)$, it follows from (\ref{eqn:dimW}) that
\[\dim\, W_q(V,K) = {n+q-1 \choose q} \frac{(n-2)(n-q-3)}{q+2}\mbox{ for all }q=0,\ldots,n-4.\]

We next consider the case $m>2n-3$. If we set $\GG :=\mathrm{Gr}_2(V^\vee )\subseteq \mathbb P(\bigwedge^2V^\vee )$ then~(\ref{eq:mainequiv}) implies that $\bb{P}(K^{\perp}) \cap \GG = \emptyset$. We consider the linear projection away from $\bb{P}(K^{\perp})$
\[ \pi \colon \mathbb P(\bigwedge^2V^\vee ) \dashrightarrow \bb{P}^{m-1}\]
Since $\dim(\GG)=2n-4$ it follows that $\pi(\GG)$ is a subvariety of $\bb{P}^{m-1}$ of codimension at least $m-1 - (2n-4) = m-2n+3$, hence a general linear subspace $\bb{P}^{m-2n+2} \subset \bb{P}^{m-1}$ is disjoint from $\pi(\GG)$, i.e. it satisfies $\pi^{-1}(\bb{P}^{m-2n+2}) \cap \GG = \emptyset$. Since $\pi^{-1}(\bb{P}^{m-2n+2})$ is a linear subspace of $\bb{P}(\bigwedge^2V^\vee )$ of codimension $2n-3$ containing $\bb{P}(K^{\perp})$, we can write $\pi^{-1}(\bb{P}^{m-2n+2})=\bb{P}(K'^\perp)$ for some $K'\subseteq K$ with $\dim(K')=2n-3$. It follows from \eqref{eq:wnat} (with the roles of $K$ and $K'$ exchanged) that $\dim\, W_q(V,K)\leq\dim\, W_q(V,K')$ for all $q$, which combined with the case $m=2n-3$ proves the desired inequality and completes our proof.
\end{proof}

\begin{rmk}
\label{rem:bigu}
A special case of Theorem~\ref{thm:HilbSeries} is obtained by taking $V = \Sym^{n-1}\bb{C}^2$ and $K=\Sym^{2n-4}\bb{C}^2$, in which case $W(V,K)$ is the \emph{Weyman module} $W(n-1)$ (see \cite[Section~5]{PS15} and \cite[Section~3.I.B]{E91}). It follows that Theorem~\ref{thm:HilbSeries} computes the Hilbert series of the said module, answering a problem left open in \cite[Section~1.2]{PS15}. For $n\le 8$, Alex Suciu used a computer search to construct subspaces $K_n \subseteq \bigwedge^2 \mathbb{C}^n$ with $m=2n-3$ and $\mathcal{R}(\mathbb{C}^n, K_n)= \{0\}$. He also computed the Hilbert series of the Koszul modules $W(\mathbb{C}^n, K_n)$. He conjectured the existence of a family of subspaces $K_n \subseteq \bigwedge^2 \mathbb{C}^n$ with $m=2n-3$ and $\mathcal{R}(\mathbb{C}^n, K_n)= \{0\}$, for all values of $n$, with Hilbert series given by the formula from Theorem~\ref{thm:HilbSeries}. Our results prove that the expectations were correct.
\end{rmk}

\section{Topological invariants of groups}
\label{subsec:top}

In this section we explain the topological motivation for investigating Koszul modules, and derive a number of consequences of our results to the study of invariants of fundamental groups. Some of our main applications concern groups $G$ that are $1$-formal (in the sense of Sullivan \cite{Su77}), where a direct relationship between certain invariants of $G$ and those of the corresponding Koszul modules can be derived. The class of $1$-formal groups contains important examples such as the fundamental groups of compact K\"ahler manifolds, hyperplane arrangement groups, or the Torelli group of the mapping class group of a genus $g$ surface.

\subsection{The Koszul module and the resonance variety of a group G}
\label{subsec:WG-RG}

We recall from Example~\ref{ex:W(G)} the construction of the \emph{Koszul module $W(G)$} associated to a finitely generated group $G$. The module $W(G):=W(V,K)$ is obtained by setting $V:= H_1(G, \mathbb{C})$, and letting $K:=\partial_G\left(H_2(G, \mathbb{C})\right)$, where $\partial_G$ is the dual of the cup product map
\[\cup_{G} \colon \bigwedge^2 H^1(G, \mathbb{C}) \to H^2(G, \mathbb{C}).\]
To the Koszul module $W(G)$ we can associate its \emph{resonance variety $\mc{R}(G)$} as in Definition~\ref{def:resonance}. Alternatively, we can define $\mc{R}(G)$ without reference to Koszul modules as follows. For an element $a\in H^1(G,\mathbb C)$, we consider the complex of abelian groups
$$(G,a):\qquad 0 \longrightarrow H^0(G,\mathbb C)\stackrel{\cup a}\longrightarrow H^1(G,\mathbb C)\stackrel{\cup a}\longrightarrow H^2(G,\mathbb C)\stackrel{\cup a} \longrightarrow \cdots.$$
The resonance variety $\mc{R}(G)$ is then obtained as
$$\cR(G)=\bigl\{a\in H^1(G,\mathbb C): H^1(G,a)\neq 0 \bigr\}.$$

In the $1$-formal context, $\mc{R}(G)$ has yet another interpretation given as follows.

\begin{rmk}\label{rem:jump-loci} For a character $\rho\in \widehat{G}:=\mbox{Hom}(G,\mathbb C^*)$, we denote by $\mathbb C_{\rho}$ the $\mathbb C[G]$-module $\mathbb C$ with multiplication given by
$g\cdot z:=\rho(g)z$. The resonance variety $\cR(G)$ turns out to be an infinitesimal analogue of the much studied \emph{characteristic variety}
$$\mathcal{V}(G):=\bigl\{\rho\in \widehat{G}: H^1(G,\mathbb C_{\rho})\neq 0\bigr\},$$
which already appeared in the work of Green--Lazarsfeld \cite{GL87}, Beauville \cite{Be}, or Libgober \cite{Li}. It is shown in \cite[Theorems A and C]{DPS} that when $G$ is $1$-formal, $\cR(G)$ is isomorphic to the tangent cone at $1$ of the characteristic variety~$\mathcal{V}(G)$.
\end{rmk}

Reinterpreted in the context of groups, our Theorems~\ref{thm=main} and~\ref{thm:HilbSeries} imply the following. We write $b_1(G) := \dim\, H_1(G,\bb{C})$ for the \emph{first Betti number} of the group $G$.

\begin{thm}\label{thm:WG-RG}
 Let $G$ be a finitely generated group and let $n=b_1(G)$. If $n\geq 3$ then we have the equivalence
 \[\mc{R}(G)=\{0\} \Longleftrightarrow W_{n-3}(G)=0.\]
 Moreover, if the equivalent statements above hold then
\[ \dim\, W_q(G) \leq {n+q-1 \choose q} \frac{(n-2)(n-q-3)}{q+2}\quad\mbox{ for }0\leq q\leq n-4.\]
\end{thm}

\vskip 3pt

This theorem can be used to obtain upper bounds on the Chen ranks of a group $G$ with vanishing resonance, once we establish a relationship between the said ranks and the Hilbert function of $W(G)$. This is achieved via the study of the Alexander invariant of $G$ in the next section.

\subsection{Alexander invariants.}
\label{subsec:alexander}
Alexander-type invariants have a long history in topology, starting with the definition of the Alexander polynomial of knots and links \cite{Alex}.
For a connected CW--complex $X$ with fundamental group $G= \pi_1(X)$, each surjective group homomorphism $\alpha\colon G \to H$ gives rise to a sequence of
Alexander-type invariants of $X$, namely the twisted homology groups $H_i(X, L)$, where $L$ is the local system $\mathbb{C}[H]$.
By construction, these invariants are $\mathbb{C}[G]$--modules. When $i=1$ the construction depends solely
on $G$, is purely algebraic and can be carried out for any group. The classical Alexander invariant of a knot $Y\subseteq S^3$ corresponds to the case
when $X:=S^3\setminus Y$ is the knot complement, $i=1$ and  $\alpha:G \to \mathbb{Z}$ is the abelianization map.

\vskip 3pt

\begin{defn}\label{alexinvarians}
If $G$ is a finitely generated group, its \emph{Alexander invariant} is defined as $B(G):=H_1(G',\mathbb Z)=G'/G''$, viewed as a module over the group ring
$\mathbb Z[G_{\mathrm{ab}}]$, where $G_{\mathrm{ab}}:=G/G'$, the action being given by conjugation. We write $B(G)_{\mathbb C}:=B(G)\otimes_{\mathbb Z} \mathbb C$.
\end{defn}

Geometrically, if $X$ is a CW--complex with $\pi_1(X)\cong G$ and $X_{\mathrm{ab}} \rightarrow X$ is the maximal abelian cover (whose group of deck transformations is $G_{\mathrm{ab}}=H_1(X,\mathbb Z)$),
then $B(G)$ can be identified with $H_1(X_{\mathrm{ab}},\mathbb Z)$, viewed as a module over $G_{\mathrm{ab}}$.

\vskip 3pt

The Alexander invariant is equipped with the $I$--adic filtration, where
$I:=\mbox{Ker}(\epsilon)$ is the kernel of the augmentation map $\epsilon\colon \mathbb Z[G/G'] \rightarrow \mathbb Z$.
A key result of Massey \cite{Massey} asserts that for each $q\geq 0$, one has an isomorphism
\begin{equation}\label{eq:Massey}
\mathrm{gr}_q B(G):=\mathrm{gr}_q^{I} B(G)\cong \mathrm{gr}_{q+2} (G/G'').
\end{equation}
In particular, the generating series of the Chen ranks of $G$ is given as a Hilbert series $$\sum_{q=0}^{\infty} \theta_{q+2}(G)t^q=\mbox{Hilb}\bigl(\mbox{gr } B(G)_{\mathbb C},t\bigr).$$

\vskip 4pt

\begin{rmk}
The Chen ranks  are usually more easily computable than the largely inaccessible lower central ranks of a group $G$. An instance of this is the case of hyperplane
arrangement groups associated to an arrangement $\mathcal{A}=\{H_1, \ldots, H_{n}\}$ of hyperplanes in $\mathbb C^m$. Whereas the fundamental group $G(\mathcal{A}):=\pi_1\bigl(\mathbb C^m-\cup_{H\in \cA} H\bigr)$ of the arrangement is \emph{not} determined by the intersection lattice of $\mathcal{A}$, the lower central series and the corresponding ranks $\phi_q(G(\mathcal{A}))$ and $\theta_q\bigl(G(\mathcal{A}))$  are, and in particular they are combinatorial objects. An explicit formula for $\phi_q(G(\cA))$ remains elusive, but the Chen ranks are subject to a conjecture of Suciu, proven in many cases in \cite{ScSu} and in full generality in \cite{CSch}. It asserts that for $q\gg 0$, one has
the formula
$$\theta_q\bigl(G(\cA)\bigr)=\sum_{r\geq 1} h_r \theta_q(F_{r+1}),$$
where $h_r$ is the number of $r$-dimensional components of the resonance variety $\mathcal{R}(G(\cA))$ and the Chen ranks of the free group are given by (see \cite{Chen})
\[\theta_1(F_{r+1})=r+1\quad\mbox{ and }\quad\theta_q(F_{r+1})=(q-1){q+r-1\choose q}\mbox{ for }q\geq 2.\]
\end{rmk}

\vskip 4pt

The Alexander invariants of $G$ exhibit the same type of behaviour as their cousins, the topological jump loci $\mathcal{V}(X):=\mathcal{V}(\pi_1(X))$ from Remark~\ref{rem:jump-loci}. They are complicated objects encoding useful information, with complexity bounded by invariants depending on the cohomology ring. The following proposition, included to make the paper more self-contained,  summarizes work from \cite{DHP, Massey, PS-imrn, PS-johnson}.

\begin{proposition}\label{comparison}
For each finitely generated group $G$ one has the inequality 
\begin{equation}
\label{eq=bw}
\mbox{dim } \mathrm{gr}_q B(G)_{\mathbb C} \leq \mbox{dim } W_q(G),
\end{equation}
with equality when the group $G$ is $1$--formal.
\end{proposition}
\begin{proof}
Denoting by $\partial_G\colon  H_2(G, \mathbb C)\rightarrow  \bigwedge^2 H_1(G,\mathbb C)$ the dual of the cup product map $\cup_G$, one defines following Chen \cite{Chen} the \emph{holonomy Lie algebra} $$\mathfrak{H}(G):=\mathbb L\bigl(H_1(G,\mathbb C)\bigr)/\mbox{Ideal}\bigl(\mbox{Im}(\partial_G)\bigr),$$
where $\mathbb L\bigl(H_1(G,\mathbb C)\bigr)$ is the free Lie algebra over $\bb{C}$ generated by $H_1(G,\mathbb C)$.
Note that $\mathbb L^1\bigl(H_1(G,\mathbb C)\bigr)=H_1(G, \mathbb C)$ and $\mathbb L^2\bigl(H_1(G,\mathbb C)\bigr)=\bigwedge^2 H_1(G, \mathbb C)$.

\vskip 3pt

Following the discussion in \cite[Section~5.4]{PS-johnson}, by the universal property of the free Lie algebra, there exists a natural surjection
$\mathbb L\bigl(H_1(G, \mathbb C)\bigr)\onto \op{gr}(G)  \oo \bb{C}$ compatible with taking derived quotients, and thus inducing a surjection
\begin{equation}\label{eq:hh''-to-GG''}
 \mathfrak{H}(G)/\mathfrak{H}''(G) \onto \op{gr}(G/G'') \oo \bb{C}.
\end{equation}
We note that the above surjection is the one given in \cite[(5.8)]{PS-johnson},  obtained as a consequence of \cite[(5.7)]{PS-johnson}, which in turn follows from \cite[Proposition~3.3]{MaP} (by taking $A^* = H^*(G)$ to be the cohomology ring of $G$, and letting $S$ be an Eilenberg--MacLane space $K(G,1)$).

Both the source and the target of \eqref{eq:hh''-to-GG''} are Lie algebras generated in degree one: for any such Lie algebra
$\mf{g}=\mf{g}_1 \oplus \mf{g}_2 \oplus \cdots $ 
we have that
\[ \mf{g}' = \mf{g}_{\geq 2} = \mf{g}_2 \oplus \mf{g}_3 \oplus \cdots\]
In particular, we get from \eqref{eq:hh''-to-GG''} a surjection
\begin{equation}\label{eq:surj-deg-2}
\mathfrak{H}'(G)/\mathfrak{H}''(G) \onto \op{gr}_{\geq 2}(G/G'') \oo \bb{C}.
\end{equation}
It is shown in \cite[Theorem~6.2]{PS-imrn} that $\mathfrak{H}'(G)/\mathfrak{H}''(G)$ has the following presentation as a module over $S:=\mbox{Sym } H_1(G,\mathbb C)$
\[ \bigl(\bw^3 V \oplus K\bigr) \oo S \lra \bw^2 V \oo S \lra \mathfrak{H}'(G)/\mathfrak{H}''(G) \lra 0,\]
where $K:=\mbox{Im}(\partial_G)$ and $\bw^2 V$ are placed in degree $2$, and $\bw^3 V$ is placed in degree $3$. Comparing this to the definition (\ref{eq:def-Koszul-module}) of the Koszul module $W(G)$ which has an equivalent presentation up to a shift in degree by two, we conclude that
\[ W_q(G) \cong \op{gr}_{q+2}\left(\mathfrak{H}'(G)/\mathfrak{H}''(G)\right)\text{ for }q\geq 0.\]
Combining this with \eqref{eq:surj-deg-2} and \eqref{eq:Massey}, we obtain that $\mbox{dim } \mbox{gr}_{q+2}(G/G'')\leq \mbox{dim } W_q(G)$. The fact that equality holds in the $1$-formal case is the content of \cite[Theorem~5.6]{DPS}.
\end{proof}

\vskip 4pt

In the following applications we relate the $I$--adic filtrations of the Alexander invariant $B(G)$  of $G$
to the cohomology ring of $G$ in low degrees in a more precise way.

\begin{thm}
\label{thm=maingroups}
Let $G$ be a finitely generated group with $b_1(G)\geq 3$. If $\mathcal{R} (G)=\{0\}$, then
\[\bigl(I^q \cdot B(G)\bigr)_{\bb{C}}=\bigl(I^{q+1} \cdot B(G)\bigr)_{\bb{C}}\]
for all $q\ge b_1(G)-3$.
\end{thm}

\proof
Combine Theorem \ref{thm:WG-RG} with \eqref{eq=bw}.
\endproof

We recall that a module $M$ over a group ring is nilpotent if $I^q \cdot M=0$ for some $q$, where $I$ is the augmentation ideal.

\begin{cor}
\label{cor=nilpmod}
If $\mathcal{R} (G)=\{0\}$ and $B(G)$ is nilpotent, then $\bigl(I^{b_1(G)-3} \cdot B(G)\bigr)_{\bb{C}}=0$. If moreover $G$ is $1$--formal, then
$\dim \, B(G)_{\mathbb C}= \dim \, W(G)< \infty$.
\end{cor}


In general, the Alexander invariant of a group may not be nilpotent as a module, but this is the case when the group itself is nilpotent. Recall that a group $G$ is said to be nilpotent if there exists $c\geq 1$ with $\Gamma_c(G)=\{1\}$. The largest $c$ such that $\Gamma_c(G)\neq \{1\}$ is called the \emph{nilpotency class} of $G$, and is denoted $\nc(G)$. A group $G$ is \emph{virtually nilpotent} if there exists a finite index nilpotent subgroup $H\leq G$. We define the \emph{virtual nilpotency class} $\vnc(G)$ of such a group $G$ to be the following quantity
\begin{equation}\label{eq:def-vnc}
\vnc(G) := \min\bigl\{\nc(H):H\leq G\mbox{ is a nilpotent finite index subgroup}\bigr\}.
\end{equation}
Since every finitely generated virtually nilpotent group $G$ contains a torsion free finite index subgroup, and since the nilpotency class of a torsion free nilpotent group equals that of any of its finite index subgroups, it follows that 
\begin{equation}\label{eq:vncG=ncH}
\vnc(G) = \nc(H)\mbox{ for any torsion free nilpotent subgroup }H\leq G\mbox{ of finite index}.
\end{equation}
One useful bound for the virtual nilpotency class of $G$ is obtained as follows.

\begin{lemma}\label{lem:vnc}
If $G$ contains a torsion-free finite index subgroup $H$ and $\Gamma_c(G)$ is finite, then $\vnc(G)=\nc(H)\leq c-1$.
\end{lemma}

\begin{proof}
Since $\Gamma_c(H)\leq\Gamma_c(G)$ and $\Gamma_c(H)$ is torsion-free, it follows that $\Gamma_c(H)=\{1\}$ and therefore $\nc(H)\leq c-1$. The equality $\vnc(G)=\nc(H)$ follows from (\ref{eq:vncG=ncH}).
\end{proof}

\begin{cor}
\label{cor=nilpgr}
Let $G$ be a finitely generated nilpotent group with vanishing resonance. Then $\bigl(I^{b_1(G)-3} \cdot B(G)\bigr)_{\bb{C}}=0$. It also follows that $\vnc(G/G'')\leq b_1(G)-2$.
\end{cor}

\begin{proof}
 Since $G$ is nilpotent, its Alexander invariant is also nilpotent so the conclusion $\bigl(I^{b_1(G)-3} \cdot B(G)\bigr)_{\bb{C}}=0$ follows from Corollary~\ref{cor=nilpmod}. It follows moreover from (\ref{eq:Massey}) and \eqref{eq=bw} that $\theta_q(G)=0$ for $q\geq b_1(G)-1$. Since $G/G''$ is nilpotent, we must have that $\Gamma_q(G/G'')$ is finite for $q\geq b_1(G)-1$. Furthermore, since $G$ is finitely generated, it contains a torsion-free finite index subgroup, so we can apply Lemma \ref{lem:vnc} to conclude that $\vnc(G/G'')\leq b_1(G)-2$.
\end{proof}

The conclusion of Corollary~\ref{cor=nilpgr} \emph{does} require the vanishing resonance assumption.

\begin{ex}
\label{ex=freenilp}
Indeed, let $F_n$ be the free group on $n\ge 2$ generators, and consider the nilpotent group $G:= F_n/\Gamma_k(F_n)$, where $k\geq 3$.
As explained for instance in \cite[Remark~2.4]{MP}, the resonance variety of a group depends only on its third nilpotent quotient.
Since $G/\Gamma_3(G)\cong F_n/\Gamma_3(F_n)$, it follows that $\mathcal{R} (G)\cong \mathcal{R} (F_n)\cong \mathbb{C}^n$. Note also that $b_1(G)=n$.
On the other hand, as already pointed out, the Hilbert series
$\sum_{q\ge 1} \theta_q(F_n) t^q$ was computed by Chen \cite{Chen} and {\em all} of its coefficients are strictly positive. Since $\theta_q(F_n)=\theta_q(G)$ for $q<k$, it follows from (\ref{eq:Massey}) that $I^q\cdot B(G)\neq 0$ for $q<k-2$, thus the nilpotency class cannot be bounded solely in terms of $n=b_1(G)$.
\end{ex}

In fact the resonance variety of a nilpotent group of nilpotency class $2$ can be arbitrarily complicated, cf. \cite[Remark 2.4]{MP}. However, this phenomenon does not occur if we assume $1$-formality. The resonance variety of a finitely generated, nilpotent, $1$-formal group $G$ must vanish,  see \cite[Lemma 2.4]{CT} and the proof of Theorem~\ref{thm:bound-chen} below. 

\begin{thm}\label{thm:bound-chen}
 Let $G$ be a finitely generated $1$-formal group and suppose its first Betti number is $n=b_1(G)\geq 3$. If $G/G''$ is nilpotent, then $\theta_q(G)=0$ for $q\geq n-1$ and
 \[\theta_q(G) \leq {n+q-3\choose n-1} \cdot \frac{(n-2)(n-1-q)}{q}\quad\mbox{ for }q=2,\ldots,n-2.\]
\end{thm}

\begin{proof} The condition that $G/G''$ is nilpotent implies that $\theta_q(G)=\phi_q(G/G'')=0$ for $q\gg 0$. Since $G$ is $1$-formal, we have using (\ref{eq:Massey}) and \eqref{eq=bw} that $\dim\,W_q(G)=\theta_{q+2}(G)$ for all $q\geq 0$. In particular $W(G)$ is a finite dimensional module, hence $\mc{R}(G)=\{0\}$ and Theorem~\ref{thm:WG-RG} applies to give the desired conclusions about $\theta_q(G)=\dim\,W_{q-2}(G)$.
\end{proof}

We proved in Theorem \ref{thm=main} that $W_{n-3}(V,K)=0$ when $\mathcal{R} (V,K)=\{ 0\}$, and therefore $W_q(V,K)=0$ for all $q\ge n-3$.
In the borderline case $m=2n-3$, we infer from Proposition \ref{prop:nonzero} that this is the best possible vanishing result
for Koszul modules. In the general case, the following examples coming from nilpotent groups show that the bound $q(n,m)= n-3$ may be improved, at least in some cases.

\begin{ex}
\label{ex=heis}
For $k\ge 1$, we denote by $H_k$ the fundamental group of the {\em Heisenberg nilmanifold} of dimension $2k+1$.
It is well-known that $H_k$ is finitely generated and $\Gamma_3(H_k)= \{ 1\}$. We set
$V:= H_1(H_k, \mathbb{C})$ and denote by $K^{\perp}\subseteq \bigwedge^2 V^{\vee}$ the kernel of the cup product map as before. If $k\ge 2$,
we know from \cite{Mac} that $H_k$ is $1$--formal and $\mathcal{R}(H_k)=\{ 0\}$. In this case, $b_1(H_k)=2k$ and $K\subseteq \bigwedge^2 V$
has codimension $1$. The fact that $\Gamma_3(H_k)= \{ 1\}$ and \eqref{eq=bw} together imply that $W_1(H_k)=0$, but note that $1<b_1(H_k)-3$ as soon as $k>2$.
\end{ex}

\subsection{A bound on the degree of growth of a group.}\label{subsec:deg-growth}

The goal of this section is to provide a proof of Theorem~\ref{thm:growth-degree}. Suppose $G$ is a finitely generated group and let $S\subseteq G$ be a finite set of generators. For an integer $m\geq 0$, we denote by $\ell_G(m)=\ell_{G,S}(m)$ the number of elements $g\in G$ expressible as a product of $m$ elements from $S\cup S^{-1}$. Bass \cite{bass} and Guivarc'h \cite{guivarch} showed that a finitely generated nilpotent group $G$ has polynomial growth. More precisely, the inequality $\ell_G(m)\leq C m^{d(G)}$
holds for $m\gg 0$, where $d(G)$ is the \emph{degree of polynomial growth} of the group. In the case of a nilpotent group $N$, this invariant is given by the  formula

\begin{equation}\label{thm:bass-guivarch}
d(N) = \sum_{q\geq 1} q\cdot\phi_q(N),
\end{equation}
where $\phi_q(N)$ denote the lower central ranks of $N$.
We also recall a celebrated theorem of Gromov \cite{gromov} which asserts that conversely, the virtually nilpotent groups are precisely the ones that have polynomial growth. If $G$ is merely virtually nilpotent, the Bass-Guivarc'h formula (\ref{thm:bass-guivarch}) fails as seen for instance by taking $G$ to be the infinite dihedral group, where $d(G)=1$, whereas $\phi_q(G)=0$ for all $q$. For a quantitative version of Gromov's result, we refer to \cite[Theorem 4]{M}.

\vskip 3pt

Before stating the following result, we recall that the \emph{Hirsch index} of a finitely generated nilpotent (or more generally polycyclic) group $G$ is defined to be the quantity
$h(G):=\sum_{q\geq 1} \phi_q(G)$, see \cite[Chapter~1]{segal}. It is well-known that for a normal subgroup $N\trianglelefteq G$ one has the additivity property $h(G)=h(N)+h(G/N)$.

\begin{lem}\label{lem:ineq-sumphi}
 Let $G$ be a finitely generated group, and $H$ be a finite index subgroup of~$G$. For every $k\geq 1$ we have that
 \begin{equation}\label{eq:sum-phi-ineq}
 \sum_{q=1}^k \phi_q(H) \geq \sum_{q=1}^k \phi_q(G).
 \end{equation}
 Moreover, if $[H\cap\Gamma_{k+1}(G):\Gamma_{k+1}(H)]<\infty$, then the above inequality is in fact an equality.
\end{lem}

\begin{proof} The quotient $\ol{H}:=H/\Gamma_{k+1}(H)$ is nilpotent and its Hirsch number is given by
\[h(\ol{H}) = \sum_{q=1}^k \phi_q(H).\]
A similar formula holds for $\ol{G}:=G/\Gamma_{k+1}(G)$. Since $H\leq G$, we must also have that $\Gamma_{k+1}(H)\leq\Gamma_{k+1}(G)$, and therefore we obtain an induced group homomorphism
\[ \psi : \ol{H} \lra \ol{G}.\]
Since $[G:H]<\infty$,  the image of $\psi$ has finite index in $\ol{G}$, therefore $h(\op{Im}\psi) = h(\ol{G})$. Since $\op{Im}\psi$ is a quotient of $\ol{H}$, it follows that $h(\ol{H})=h(\ker\psi) + h(\op{Im}\psi)\geq h(\op{Im}\psi)$, which proves that $h(\ol{H})\geq h(\ol{G})$. If we assume that $\Gamma_{k+1}(H)$ has finite index in $H\cap\Gamma_{k+1}(G)$ then $\ker\psi$ is a finite group and therefore $h(\ker\psi)=0$, which implies $h(\ol{H})=h(\ol{G})$.
\end{proof}

In the next proof, recall that $M:=G/G''$ denotes the maximal metabelian quotient of a $1$-formal finitely generated group $G$.

\begin{proof}[Proof of Theorem~\ref{thm:growth-degree}]
Since $\Gamma_q(M)$ is finite for $q\gg 0$, it follows that $\theta_q(G)=\phi_q(M)=0$ for $q\gg 0$. Using (\ref{eq:Massey}) and the fact that (\ref{eq=bw}) is an equality when $G$ is $1$-formal, we obtain
\[\dim\, W_q(G) = \phi_{q+2}(M) = 0 \mbox{ for }q\gg 0.\]
It follows that $W(G)$ is finite dimensional, which by (\ref{eq:mainequiv}) and Theorem~\ref{thm=main} implies that $W_q(G)=0$ for $q\geq n-3$ and thus $\phi_{q}(M)=0$ for $q\geq n-1$. Combining this with the fact that $\Gamma_q(M)$ is finite for $q\gg 0$, we conclude that in fact $\Gamma_{n-1}(M)$ is finite proving conclusion (\ref{it:conc-1}) of Theorem~\ref{thm:growth-degree}.

 We note that if $N$ is a finite normal subgroup of $M=G/G''$ then the groups $M$ and $M/N$ have the same growth rate. Taking $N=\Gamma_q(M)$ where $q$ is such that $\Gamma_q(M)$ is finite, we get that $M$ has the same growth rate as the nilpotent group $M/N$, and in particular $M$ is virtually nilpotent by Gromov's theorem. We can then apply Lemma~\ref{lem:vnc} to deduce part (\ref{it:conc-2}) of Theorem~\ref{thm:growth-degree}.

Consider now $H$ to be a nilpotent finite index subgroup of $M$. Using the last part of Lemma~\ref{lem:ineq-sumphi} and the fact that $\Gamma_{n-1}(M)$ is finite we get that
\begin{equation}\label{eq:sum-phi-equal}
\sum_{q=1}^{n-2} \phi_q(H) = \sum_{q=1}^{n-2} \phi_q(M).
\end{equation}
Applying the Bass--Guivarc'h Formula \ref{thm:growth-degree} to $H$ yields
\[d(M) = d(H) = \sum_{q\geq 1} q\cdot \phi_q(H) = (n-2)\cdot\left(\sum_{q=1}^{n-2} \phi_q(H)\right) - \sum_{k=1}^{n-3}\left(\sum_{q=1}^k \phi_q(H)\right)\]
\[\overset{(\ref{eq:sum-phi-ineq}),(\ref{eq:sum-phi-equal})}{\leq} (n-2)\cdot\left(\sum_{q=1}^{n-2} \phi_q(M)\right) - \sum_{k=1}^{n-3}\left(\sum_{q=1}^k \phi_q(M)\right) = \sum_{q = 1}^{n-2} q\cdot \phi_q(M)\]
\[=\phi_1(M) + \sum_{q = 0}^{n-4} (q+2)\cdot \dim\, W_q(G).\]
Using the fact that $\phi_1(M)=b_1(G) = n$ and the inequality in Theorem~\ref{thm:HilbSeries}, it follows that
\[ d(M) \leq n + \sum_{q=0}^{n-4} {n+q-1\choose q} \cdot (n-2)\cdot(n-q-3) = n + (n-2)\cdot{2n-3\choose n-4}.\qedhere\]
\end{proof}

\begin{rmk}\label{rem:freenilp}
Just as in the case of the nilpotency class, in the absence of $1$-formality there can be no uniform bound in terms of $b_1(G)$ for the degree of polynomial growth of $M=G/G''$. Indeed, if we take $G= F_n/\Gamma_k(F_n)$ as in Example~\ref{ex=freenilp}, then $M$ is nilpotent. Combining~(\ref{thm:bass-guivarch}) with the fact that $\phi_q(M) = \theta_q(G) = \theta_q(F_n)$ is strictly positive for $q<k$, we obtain that $d(M)$ can be made arbitrarily large if we let $k\to\infty$.
\end{rmk}

\begin{rmk}\label{rem:vanishing-resonance}
 The only place in the proof of Theorem~\ref{thm:growth-degree} that required $1$-formality is the assertion that the vanishing of the Chen ranks implies the vanishing of $W_q(G)$. If we replace the $1$-formality assumption with the hypothesis that $\mc{R}(G)=\{0\}$ then the vanishing of $W_q(G)$ follows from Theorem~\ref{thm:WG-RG}, and the proof of Theorem~\ref{thm:growth-degree} carries through without any further modifications.
\end{rmk}

\subsection{K\"ahler groups.}\label{kaehlercsoport} Our results can be applied to analyze the nilpotency class of fundamental groups of compact K\"ahler manifolds.
Suppose $X$ is a compact K\"ahler manifold and $G:=\pi_1(X)$, hence $H_1(G,\mathbb Z)\cong H_1(X,\mathbb Z)$. The first thing we observe is that
$\cR(G)\neq 0$ if and only if the $(2,0)$-part of the cup product map
$$\psi_X=\cup_X^{2,0}\colon \bigwedge^2 H^0(X,\Omega_X^1)\rightarrow H^0(X,\Omega_X^2)$$
is zero on a degenerate element $0\neq \omega_1\wedge \omega_2\in \bigwedge^2 H^0(X,\Omega_X^1)$. This follows because
$\cup_X^{0,2}$ is the conjugate of $\cup_X^{2,0}$, whereas the $(1,1)$-part of the cup product in the Hodge decomposition
$\cup_X^{1,1}:H^{1,0}(X)\otimes H^{0,1}(X)\rightarrow H^{1,1}(X)$ cannot vanish on decomposable tensors. Following Catanese's
generalization of the Castelnuovo-de Franchis Theorem~\cite{Cat}, it follows that $\cR(X)\neq 0$ if and only if $X$ is fibred onto
a curve $C$ of genus at least~$2$. In this case, we have a surjection $\pi_1(X)\twoheadrightarrow \pi_1(C)$, thus $\pi_1(X)$ is not nilpotent.
We say that in this case $X$ is \emph{fibred}. Since $\pi_1(X)$ is $1$-formal, we obtain as a special case of Theorem~\ref{thm:growth-degree} the following, where we recall that $q(X)=h^1(X, \mathcal{O}_X)=\frac{b_1\bigl(\pi_1(X)\bigr)}{2}$ is the irregularity of~$X$.

\begin{thm}\label{kaehler-csoport}
Suppose that $X$ is a compact K\"ahler manifold. If $\pi_1(X)/\pi_1(X)''$ is nilpotent, then its virtual nilpotency class is at most $2q(X)-2$.
\end{thm}

Constructing explicit K\"ahler manifolds with nilpotent fundamental group has proven to be a challenging task. Examples of K\"ahler groups with a non-trivial nilpotent filtration (in fact, of nilpotency class $2$) have been produced first by Campana \cite{Cam}. On the other hand, a K\"ahler group
is either fibred or each of its solvable quotients is virtually nilpotent \cite{Del}. In particular, solvable K\"ahler groups are virtually nilpotent.

\vskip 3pt

A recent example is provided by the Schoen surfaces discovered in \cite{Schoe} and studied further in \cite{CLMR}. They are minimal algebraic surfaces $X$ of general type having
invariants $$q(X)=4, \ p_g(X)=5, \ K_X^2=16, \ h^{1,1}(X)=12.$$
 Remarkably, $p_g(X)=2q(X)-3$, that is, we are in  the divisorial case when $\mbox{Ker}(\psi_X)$ may not contain decomposable elements. That is the case, for Schoen surfaces are not fibred. In fact $\mbox{Ker}(\psi_X)$ is of dimension one, implying $\mbox{dim } \mbox{Ker}(\cup_X)\in \{6,7\}$. It is unknown whether the fundamental group of such an $X$ is nilpotent. Applying Theorem \ref{kaehler-csoport}, we obtain that the virtual nilpotency class of $\pi_1(X)/\pi_1(X)''$ is at most $6=b_1(X)-2$.

\subsection{The Torelli group of the mapping class group.}\label{subsec:torelli}
A particularly interesting class of applications of our results  is provided by the {\em Torelli groups} $T_g$ of genus  $g\ge 2$.
We fix  a closed oriented surface $\Sigma_g$ of genus $g$ and denote by $\pi_g:=\pi_1(\Sigma_g)$ its fundamental group and by $H:=H_1(\Sigma_g,\mathbb Z)$ its first homology. Let $\mathrm{Mod}_g$ be the mapping class group  of isotopy classes of orientation-preserving homeomorphisms of $\Sigma_g$.  If $\mathcal{X}_g$ denotes the Teichm\" uller space of genus $g$, then  $\mathrm{Mod}_g$ is the orbifold
fundamental group of the moduli space $\cM_g$ of smooth curves  of genus $g$,  which can be realized as the quotient $\cM_g\cong \mathcal{X}_g/\mathrm{Mod}_g$.  Nielsen theory offers an alternative description
$\mathrm{Mod}_g\cong \mathrm{Out}^+(\pi_g)$ of the mapping class group as the group of orientation preserving outer isomorphisms of~$\pi_g$. The Torelli group consists of those outer automorphisms of $\pi_g$ that act trivially on~$H$. The \emph{Torelli space}
$$\mathcal{T}_g:=\mathcal{X}_g/T_g$$
can then be thought of as the moduli space of pairs $[C,\alpha_1, \ldots, \alpha_g, \beta_1, \ldots,  \beta_g]$, consisting of a smooth curve $C$ of genus $g$ and a symplectic basis of $H_1(C,\mathbb Z)$.

Johnson \cite{J} defined the surjective homomorphism
$$\tau\colon T_g\rightarrow \bigwedge^3 H/H,$$
where the $\mathrm{Sp}_{2g}(\mathbb Z)$-equivariant injective map $H\hookrightarrow \bigwedge^3H$ is given by $z\mapsto z\wedge \omega$, where $\omega:=\alpha_1\wedge \beta_1+\cdots +\alpha_g\wedge \beta_g\in \bigwedge^2 H$, with $(\alpha_1, \ldots, \alpha_g,\beta_1, \ldots, \beta_g)$ being a symplectic basis of $H$. Then he proved that $T_g$ is finitely generated and that the \emph{Johnson kernel}
defined by the exact sequence
$$1\longrightarrow K_g\longrightarrow T_g\longrightarrow \bigwedge^3 H/H\longrightarrow 1,$$
is precisely the subgroup of $\mathrm{Mod}_g$ generated by Dehn twists along separating simple closed curves in $\Sigma_g$. Finally, Johnson \cite{J2}  showed that the abelianization of $T_g$ is given (up to $2$-torsion) by the map $\tau$, that is one has an isomorphism
$$\tau_*\colon H_1(T_g,\mathbb Q)\cong \bigwedge^3 H_{\mathbb Q}/H_{\mathbb Q},$$
which in particular implies that
$$b_1(T_g) = {2g\choose 3}-2g.$$

For  $g\ge 4$, the Torelli group is $1$-formal  \cite{Hai, Hai2} and  $\mathcal{R} (T_g)=\{ 0\}$, see \cite[Theorem~4.4]{DP}. Furthermore, Dimca, Hain and Papadima \cite[Corollary~3.5]{DHP} proved that $B(T_g)_{\mathbb C}=H_1(K_g,\mathbb C)$ is a nilpotent module over $\mathbb C[T_g/T_g']$. We deduce from Corollary~\ref{cor=nilpmod} the following new estimate of nilpotency of the Alexander invariant of $T_g$.

\begin{thm}\label{thm:Alexander-Tg}
 If $g\geq 4$ and $I \subseteq\bb{C}[T_g/T_g']$ denotes the augmentation ideal, then
 \[I^q\cdot B(T_g)_{\mathbb C}=0 \mbox{ for } q\geq {2g\choose 3}-2g-3.\]
\end{thm}

More recently, it was shown in \cite[Corollary 1.5]{EH} that the metabelian quotient $T_g/T_g''$ is nilpotent for $g\geq 12$, and in particular it follows that Theorem~\ref{thm:growth-degree} applies for the Torelli groups. It follows that $\vnc(T_g/T_g'')\leq b_1(T_g)-2$, which is the statement of Theorem~\ref{thmB} in the Introduction. Of course, the groups $T_g$ and $\mathrm{Mod}_g$ are very far from being nilpotent. In fact, if $G\leq \mathrm{Mod}_g$ is a subgroup, either $G$ is virtually abelian, or it contains a non-abelian free subgroup, see \cite{McC}.

\vskip 3pt

\subsection{The Torelli group of a free group.} More generally, for  an arbitrary finitely generated group $G$, we  define its \emph{Torelli group} $T(G)\subseteq \mathrm{Out}(G)$ to be the subgroup of
outer automorphisms of $G$ consisting of automorphisms which induce the identity on $H_1(G, \mathbb{Z})$. Of particular interest is the case when $G=F_g$ is a free group on $g$ generators $x_1, \ldots, x_g$.
The corresponding Torelli group $T(F_g)$ denoted by $\OA_g$, is thus defined by the following exact sequence
$$1\longrightarrow \OA_g\longrightarrow \mbox{Out}(F_g)\longrightarrow GL_g(\mathbb Z)\longrightarrow 1.$$

The Culler-Vogtmann \emph{outer space} $X_g$ of marked metric graphs of rank $g$ can be thought of as the classifying space for $\OA_g$, see \cite{CV}. The moduli space of graphs $X_g/\OA_g$ is (essentially) the moduli space $\cM_g^{\mathrm{tr}}$ of tropical curves of genus $g$, see \cite{CMV} and references therein.
In analogy with $\OA_g$, one defines the subgroup $\IA_g$ of the group $\mathrm{Aut}(F_g)$ of all automorphisms of the free group which induce the identity on $H_1(F_g, \mathbb{Z})$. Both $\IA_g$ and $\OA_g$ are finitely generated for $g\geq 3$ by \cite{Mag}, and there exist natural quotient maps $\IA_g\onto\OA_g$.

\vskip 3pt

It is known that for $g\ge 4$ one has $\mathcal{R} (\OA_g)=\{ 0\}$, see \cite[Theorem~9.7]{PS-johnson}. Moreover, the first Betti number of $\OA_g$ has been computed by Andreadakis \cite{Andr} and Bachmut \cite{Bach}, see also \cite{K}, \cite[Theorem~9.1]{PS-johnson} and the references therein, and is given by
$$ b_1(\OA_g) = \frac{g(g+1)(g-2)}{2}.$$
Note \cite{Mag} that $\IA_g$ is generated by the following set of \emph{Magnus generators} of $\mbox{Aut}(F_g)$
$$\alpha_{ij}: x_i\mapsto x_jx_ix_j^{-1}, \ \ x_{\ell}\mapsto x_{\ell} \ \mbox{for } \ell\neq i, \mbox{ and }$$
$$\alpha_{ijk}: x_i\mapsto x_i\cdot [x_j,x_k], \ \ x_{\ell}\mapsto x_{\ell} \mbox{ for } \ell\neq i,$$
where the indices are subject to the conditions $1\leq i\neq j\leq g$ in the first case, respectively $1\leq j<k\leq g$ and $i\neq j,k$, in the second case. One has
$\frac{g^2(g-1)}{2}=b_1(\IA_g)$ generators, see again \cite{K}. When passing to $\OA_g$, one notices the relation $\alpha_{1,\ell} \cdots  \alpha_{g,\ell}=1$ for $\ell=1, \ldots, g$ and one is left with $b_1(\IA_g)-g=\frac{g(g+1)(g-2)}{2}$ generators, which explains the formula for $b_1(\OA_g)$.

\vskip 3pt

It is shown in \cite[Corollary 1.5]{EH} that the metabelian quotients of $\IA_g$ are nilpotent for $g\geq 4$, so the same is true for the metabelian quotients of $\OA_g$. Applying Theorem~\ref{thmA} with $G=\OA_g$ we obtain the following result.
\begin{thm}\label{iag}
\label{thm=oan}
If $g\geq 4$ then the metabelian quotient $\OA_g/\OA_g''$ has virtual nilpotency class at most $\frac{g(g+1)(g-2)}{2}-2$.
\end{thm}

Using Corollary~\ref{cor=nilpmod} we also obtain an analogue of Theorem~\ref{thm:Alexander-Tg} for the Alexander invariant of $\OA_g$.

\end{document}